\documentclass{article}

\usepackage{amsthm}
\usepackage{amssymb}
\usepackage{amsmath}
\usepackage{amsfonts}
\usepackage{latexsym}
\usepackage{url}
\usepackage[dvips]{epsfig}

\theoremstyle{plain}
\newtheorem{thm}{Theorem}[section]
\newtheorem{lem}[thm]{Lemma}
\newtheorem{prop}[thm]{Proposition}

\theoremstyle{definition}

\newcommand{\br}[2][q]{[#2]_{#1}}
\newcommand{\s}{\sigma}
\newcommand{\qyeul}[2]{E_{#1, #2}(q,y)}
\newcommand{\N}{\mathbb{N}}

\newcommand{\inv}{\mathrm{INV}}
\newcommand{\maj}{\mathrm{MAJ}}
\newcommand{\den}{\mathrm{DEN}}
\newcommand{\mak}{\mathrm{MAK}}
\newcommand{\sor}{\mathrm{SOR}}
\newcommand{\mstat}{\mathrm{MST}}

\newcommand{\des}{\mathrm{des}}
\newcommand{\exc}{\mathrm{exc}}

\newcommand{\estat}{\mathrm{est}}
\newcommand{\stc}{\mathrm{stc}}

\newcommand{\rlmin}{\mathrm{rlmin}}
\newcommand{\cyc}{\mathrm{cyc}}

\newcommand{\sstat}{\mathrm{sst}}

\newcommand{\invcode}{\mathrm{INVcode}}
\newcommand{\majcode}{\mathrm{MAJcode}}
\newcommand{\dencode}{\mathrm{DENcode}}
\newcommand{\hancode}{\mathrm{HANcode}}
\newcommand{\makcode}{\mathrm{MAKcode}}
\newcommand{\sorcode}{\mathrm{SORcode}}
\newcommand{\mstatcode}{\mathrm{MSTcode}}
\newcommand{\add}{\mathrm{sum}}
\newcommand{\codemap}{\mathcal{C}}

\newcommand{\rem}{\mathrm{Rem}}
\newcommand{\zer}{\mathrm{zer}}
\newcommand{\st}{\mathrm{st}}

\title{Stirling-Euler-Mahonian Triples of Permutation Statistics}
\author{Frederick Butler}
\date{}

\begin{document}

\maketitle






\begin{abstract}
Two well-known distributions in the study of permutation statistics are the Mahonian and Eulerian distributions. Mahonian statistics include the major index $\maj$ and the number of inversions $\inv$, while examples of Eulerian statistics are the number of descents $\des$ and excedances $\exc$. Also of interest are pairs of permutation statistics that have the same joint distribution as $(\des, \maj)$; these are called Euler Mahonian pairs. In this paper we define a \textit{Stirling} statistic as one that is equidistributed with the number of right-to-left minima $\rlmin$ and a \textit{Stirling-Euler-Mahonian} triple as a triple of statistics with the same trivariate distribution as $(\rlmin, \des, \maj)$. We prove that several triples involving well-known statistics are Stirling-Euler-Mahonian. A general bijective technique is demonstrated for any Mahonian statistic that has a permutation code. Finally, we note several properties of Stirling-Euler-Mahonian triples that follow from earlier results. 
\end{abstract}




\section{Introduction}
Let $\mathfrak{S}_n$ denote the set of all permutations of $\{1, 2, \ldots, n\}$. A \textit{permutation statistic} is a function $\mathfrak{S}_n\to \{0, 1, 2, 3, \ldots\}$. While permutation statistics can have any distribution on $\mathfrak{S}_n$, much research has focused on the Mahonian and Eulerian distributions. 

A permutation statistic $\mstat$ is said to be \textit{Mahonian} if
\begin{equation}\sum_{\s\in \mathfrak{S}_n} q^{\mstat(\s)}=(1+q)(1+q+q^2)\cdots (1+q+\cdots +q^{n-1}).\label{eq:DefMahStat}\end{equation}
Following the convention established in~\cite{CSZ97}, we will denote Mahonian statistics in all capital letters. We give two classic examples of Mahonian statistics.

Let $\s_i$ denote the $i$th entry of the word representing the permutation $\s\in\mathfrak{S}_n$. An \textit{inversion} in $\s$ is a pair $j<k$ such that $\s_j>\s_k$. We use $\# X$ to denote the cardinality of the set $X$. The statistic $\inv$  defined as
$$\inv(\sigma)=\#\{j<k \, \vert \, \sigma_j >\sigma_k\}$$
is called the \textit{inversion number} of $\s$. The inversion number was proven to be Mahonian by Rodrigues in~\cite{Rod}.  

The \textit{major index} $\maj(\sigma)$ is given by
$$\maj(\sigma)=\sum_{\sigma_{j}>\sigma_{j+1}} j.$$
Note $\maj(\s)$ can be described as the sum of the descents of $\s$ (we define descents below). The statistic $\maj$ was proven to be Mahonian in~\cite{MacMahon} by MacMahon, after whom both the statistic and the distribution are named (MacMahon was a major in the British military). The first bijective proof of this fact was given by Foata in~\cite{Foata68}. 

A \textit{descent} of $\sigma$ is a position $j$ for which $\sigma_{j}>\sigma_{j+1}$. The value $\s_{j}$ is called the \textit{descent top} and the value $\s_{j+1}$ the \textit{descent bottom}. We'll use these terms again in Section~\ref{Sec:IntroToMAK}. The \textit{descent number}
$$\des(\sigma)=\#\{j \, \vert \, \sigma_{j}>\sigma_{j+1} \}$$ counts the number of descents in $\s$.

A permutation statistic $\estat$ is said to be \textit{Eulerian} if it has the same distribution as $\des$; that is,
\begin{equation}\sum_{\s\in \mathfrak{S}_n}x^{\estat(\s)}=\sum_{\s\in \mathfrak{S}_n}x^{\des(\s)}.\label{eq:DefEulStat}\end{equation}
The polynomials in (\ref{eq:DefEulStat}) are called the \textit{Eulerian polynomials}. While there is no known closed formula for (\ref{eq:DefEulStat}) similar to that in (\ref{eq:DefMahStat}), there is an exponential generating function whose coefficients are the Eulerian polynomials; see~\cite{foata}.

A well-known Eulerian statistic is $\exc$, which we define now. An \textit{excedance} of $\sigma$ is a position $j$ for which $\sigma_j>j$. Using similar language to that for descents, we will refer to the value $\s_j$ as the \textit{excedance top} (we will use this terminology more in Sections~\ref{Sec:Den} and~\ref{Sec:IntroToMAK}). The \textit{excedance number} $$\exc(\sigma)=\#\{j \, \vert \, \sigma_j>j\}$$ counts the number of excedances in $\s$. MacMahon first proved $\exc$ is Eulerian in~\cite{MacMahon}, while Foata offered a combinatorial proof in~\cite{Foata64}.

Eventually research turned to the study bivariate distributions. A pair of permutation statistics $(\estat, \mstat)$ is \textit{Euler-Mahonian} if it is equidistributed with $(\des, \maj)$; that is,
\begin{equation}
\sum_{\s\in\mathfrak{S}_n}x^{\estat(\s)}q^{\mstat(\s)}=\sum_{\s\in\mathfrak{S}_n}x^{\des(\s)}q^{\maj(\s)}\label{Eq:DefEulMah}
\end{equation}
Note (\ref{Eq:DefEulMah}) implies that $\estat$ must be Eulerian and $\mstat$ must be Mahonian. There are several known Euler-Mahonian pairs of statistics. These include $(\exc,\den)$, where $\den$ is \textit{Denert's statistic} described in Section~\ref{Sec:Den}; $(\des,\mak)$ with $\mak$ as defined in Section~\ref{Sec:IntroToMAK}; and $(\stc, \inv)$, where $\stc$ is discussed in Section~\ref{Sec:Inv}. The pair $(\exc,\den)$ was first proven to be Euler-Mahonian in~\cite{FZ90}; this fact was proven bijectively in~\cite{Han}. The other two pairs were proven Euler-Mahonian in~\cite{CSZ97} and~\cite{Skandera}, respectively. 

We will call a permutation statistic $\sstat$ a \textit{Stirling} statistic if
\begin{equation}\sum_{\s\in \mathfrak{S}_n} y^{\sstat(\s)}=y(y+1)(y+2)\cdots (y+n-1).\label{Def:Stirling}\end{equation}
That is, $\sstat$ has the same generating function as the unsigned Stirling numbers of the first kind. See~\cite[Proposition 1.3.7]{stanley}.

We describe several examples of Stirling statistics. The first is the \textit{number of cycles} in a decomposition of $\sigma$ into disjoint cycles, including those of length $1$, which is denoted $\cyc(\sigma)$. A value $\sigma_j$ is a \textit{right-to-left minimum} of $\sigma$ if it is smaller than all values to its right in the word for $\s$, so $\sigma_j<\sigma_k$ for all $k>j$. The statistic $\rlmin(\sigma)$ is defined to equal the number of right-to-left minima of $\sigma$. Through symmetry it is easy to show that $\rlmin$ has the same distribution on $\mathfrak{S}_n$ as the number of \textit{right-to-left maxima}, \textit{left-to-right minima}, and \textit{left-to-right maxima} (see the proof of~\cite[Proposition 1.3.1]{stanley}). It can then be established that $\cyc$ and $\rlmin$ both satisfy (\ref{Def:Stirling}) via~\cite[Corollary 1.3.11 ]{stanley}.

We will now define a \textit{Stirling-Euler-Mahonian} triple of statistics $(\sstat, \estat, \mstat)$ as one that has the same trivariate distribution as the triple $(\rlmin, \des, \maj)$; that is,
\begin{equation}\label{Def:CEM}
  \sum_{\s\in \mathfrak{S}_n}y^{\sstat(\s)}x^{\estat(\s)}q^{\mstat(\s)}=\sum_{\s\in \mathfrak{S}_n}y^{\rlmin(\s)}x^{\des(\s)}q^{\maj(\s)}.
\end{equation}

\noindent Note~(\ref{Def:CEM}) implies that $\sstat$ must be a Stirling statistic, $\estat$ must be Eulerian, and $\mstat$ must be Mahonian. What we call Stirling-Euler-Mahonian above was first introduced as \textit{cycle-Euler-Mahonian} in~\cite{CycleCounting}. We feel the revised term better reflects the more general nature of the statistic $\sstat$.

Our main goal in this paper will be to prove that several triples involving various permutation statistics are Stirling-Euler-Mahonian. We give a general technique for doing so for any Mahonian permutation statistic that has a permutation code, which we define in the next section. The remainder of this paper is organized as follows. In Section~\ref{Sec:Codes} we define permutation codes and introduce the theorems we will be using to prove the main results of this paper. In the next several sections, we will use these techniques to find Stirling-Euler-Mahonian mates for known Mahonian statistics:  the inversion number $\inv$ in Section~\ref{Sec:Inv}; Denert's statistic $\den$ in Section~\ref{Sec:Den} (where we give two such mates); Foata and Zeilberger's $\mak$ \cite{FZ90} in Section~\ref{Sec:IntroToMAK}; and Petersen and Wilson's sorting index $\sor$ in Section~\ref{Sec:IntroToSor}. Finally in Section~\ref{Sec:Properties}, we prove several properties of Stirling-Euler-Mahonian triples that follow from earlier results. 










\section{Permutation codes and $\majcode$}\label{Sec:Codes}

In this section we introduce the tools and theorems we will use to prove the main results of this paper. Let $E_n$ denote the set of all $n$-tuples of integers $(e_1, e_2, \ldots, e_n)$ satisfying $0\leq e_j\leq j-1$ for $j=1, 2, \ldots, n$. Note that $\#E_n=n!$. For any $n$-tuple $(e_1, e_2, \ldots, e_n)\in E_n$, let
$$\add((e_1, e_2, \ldots, e_n))=e_1+e_2+\cdots +e_n.$$
For a Mahonian statistic $\mstat$, a bijection $\mstatcode:\mathfrak{S}_n\to E_n$ with the property that $\add(\mstatcode(\s))=\mstat(\s)$ is called a \textit{permutation code} for $\mstat$.

We will define such a code for $\maj$ using the work of Skandera~\cite{Skandera}. What we will call the $\majcode$, Skandera refers to as the major index table. Let $\sigma\in\mathfrak{S}_n$ and define $\sigma^{(j)}$ to be the subword of $\sigma$ containing only the values $1, 2, \ldots j$. Let
$$m_j(\s)=\maj(\sigma^{(j)})-\maj(\sigma^{(j-1)}).$$
Then we can define $$\majcode(\s)=(m_1(\s), m_2(\s), \ldots,  m_n(\s)).$$
For example $\majcode(354162)=(0,0,1,1,3,5)$; see Table~\ref{table:majcode ex} for more details. Note that $\add(\majcode(354162))=\add((0,0,1,1,3,5))=10=\maj(354162)$. Skandera~\cite[Theorem 2.1]{Skandera} proves that $\majcode: \mathfrak{S}_n\to E_n$ is a bijection. It is clear from the definition that for any $\s$, $\add(\majcode(\s))=\maj(\s)$. 

\begin{table}[]
\caption{Calculating $\majcode(354162)$}
\label{table:majcode ex}
\begin{tabular}{cccc}
$j$ & $\sigma^{(j)}$ & $\maj(\sigma^{(j)})$ & $m_j(\sigma)$ \\
1   & 1              & 0                    & 0             \\
2   & 12             & 0                    & 0             \\
3   & 312            & 1                    & 1             \\
4   & 3412           & 2                    & 1             \\
5   & 35412          & 5                    & 3             \\
6   & 354162         & 10                   & 5
\end{tabular}
\end{table}

Let $\zer$ represent the number of zeros in the code of a permutation. Then we have the following proposition.

\begin{prop}
  Let $\s$ be a permutation in $\mathfrak{S}_n$. Then $\zer(\majcode(\s))=\rlmin(\s)$.\label{prop:zer=rlmin}
\end{prop}
\begin{proof}
We'll use induction on $n$. When $n=1$, $\s=1$ has $\maj$-code equal to $0$ and the result clearly holds.

Now assume the theorem holds for $n\leq k$, and let $\s\in S_{k+1}$. Consider $\majcode(\s)=(m_1(\s), m_2(\s), \ldots, m_k(\s), m_{k+1}(\s))$. Let $\pi$ be the permutation obtained by removing $k+1$ from $\s$. By the induction hypothesis, $\zer(\majcode(\pi))=\rlmin(\pi)$. It is clear from the definition that $\majcode(\pi)=(m_1(\s), m_2(\s),\ldots, m_k(\s)).$ Now replace $k+1$ in its original position to form $\s$. If $k+1$ goes into any of the positions $1, 2, \ldots, k$ of $\pi$, it will increase $\maj$ and thus we'll have $m_{k+1}(\s)\ne 0$. So $\zer(\majcode(\s))=\zer(\majcode(\pi))$. It also will not increase the number of right-to-left minima, since $k+1$ will always have at least one of the smaller values $1, 2, \ldots, k$ to its right, so $\rlmin(\s)=\rlmin(\pi)$. Thus in this case, $\rlmin(\s)=\zer(\majcode(\s))$ as desired. If $k+1$ goes into position $k+1$ of $\pi$, it will not increase $\maj$ and thus we'll have $m_{k+1}(\s)=0$. So $\zer(\majcode(\s))=\zer(\majcode(\pi))+1$. It will also increase the number of right-to-left minima, since the last entry of a permutation is always a right-to-left minimum, so $\rlmin(\s)=\rlmin(\pi)+1$. Thus in this case $\zer(\majcode(\s))=\rlmin(\s)$, and the result is true in general as desired.
\end{proof}

Skandera~\cite{Skandera} also defines a function $\st:E_n\to\{0, 1, 2, 3, \ldots\}$ such that $$\st((e_1, e_2,\ldots, e_n))=\ell,$$ where $\ell$ is the largest number such that $(e_1, e_2,\ldots, e_n)$ contains a subsequence $(e_{j_1}, e_{j_2}, \ldots, e_{j_{\ell}})$ with the property that $e_{j_k}\geq k$ for $1\leq k\leq \ell$. Note there may be several subsequences with this property. For example $\st((0,0,1,1,3,5))=3$ with $(1,3,5)$ such a subsequence. Skandera proves the following useful result (see the proof of~\cite[Theorem 3.1]{Skandera}).

\begin{prop}
  For any $\s\in \mathfrak{S}_n$, $\st(\majcode(\s))=\des(\s)$.\label{prop:st=des}
\end{prop}

We see that\ $\st(\majcode((354162)))=3=\des(354162)$. Propositions~\ref{prop:zer=rlmin} and~\ref{prop:st=des} provide our bijective method for finding Stirling-Euler-Mahonian mates for any Mahonian statistic $\mstat$ with a code $\mstatcode$, which we detail in the following theorem.
\begin{thm}
  Let $\mstat$ be a Mahonian statistic with permutation code $\mstatcode$. Then the triple $(\zer(\mstatcode),\st(\mstatcode), \mstat)$ is Stirling-Euler-Mahonian.\label{thm:SEM}
\end{thm}

\begin{proof}
Since $\mstatcode: \mathfrak{S}_n\to E_n$ and $\majcode: \mathfrak{S}_n\to E_n$ are both bijections, the map $\codemap: \mathfrak{S}_n\to \mathfrak{S}_n$ defined by
$$\codemap=(\mstatcode)^{-1}\circ \majcode$$
is also a bijection. Then $\zer(\majcode(\s))=\zer(\mstatcode(\codemap(\s))$, so $\rlmin(\s)=\zer(\mstatcode(\codemap(\s)))$ by Proposition~\ref{prop:zer=rlmin}. Also $\st(\majcode(\s))=\st(\mstatcode(\codemap(\s)))$, so $\des(\s)=\st(\mstatcode(\codemap(\s)))$ by Proposition~\ref{prop:st=des}. Finally, $\add(\majcode(\s))=\add(\mstatcode(\codemap(\s)))$, so $\maj(\s)=\mstat(\codemap(\s))$. Thus this bijection has the property that
\begin{multline*}(\rlmin(\s), \des(\s), \maj(\s))=\\
(\zer(\mstatcode(\codemap(\s)),\st(\mstatcode(\codemap(\s))), \mstat(\codemap(\s))),\end{multline*}
for all $\s\in \mathfrak{S}_n$, making $(\zer(\mstatcode),\st(\mstatcode), \mstat)$ a Stirling-Euler-Mahonian triple as desired.
\end{proof}

Note that Theorem~\ref{thm:SEM} is an extension of Corollary 3.2 of~\cite{Skandera}. We will use Theorem~\ref{thm:SEM} throughout the rest of this paper, finding familiar interpretations for $\zer(\mstatcode)$ or $\st(\mstatcode)$ whenever possible.

\section{A Stirling-Euler-Mahonian triple for $\inv$}\label{Sec:Inv}

We begin with a simple application of Theorem~\ref{thm:SEM}. Let $\sigma\in\mathfrak{S}_n$ be a permutation. For each $j$, $1\leq j\leq n$, we can let
$$i_j(\s)=\#\{k>j \vert \sigma_{j} >\sigma_k\}.$$
The value $i_j(\s)$ is counting the number of inversions in $\sigma$ for which $\sigma_{j}$ is the bigger number. Then define $\invcode(\s)=(i_n(\s), i_{n-1}(\s), \ldots, i_1(\s))$. The $i_j$ terms are written with indices in descending order so that $\invcode(\s)$ is an element of $E_n$. Note that $\invcode$ is referred to as the Lehmer code in the literature~\cite{lehmer}. For example, $\invcode(341625)=(0, 0, 2, 0, 2, 2)$.

We state the following well-known theorem and give a brief proof.
 \begin{prop}For any $\s\in \mathfrak{S}_n$, $\zer(\invcode(\s))=\rlmin(\sigma)$.\label{prop:zerinvcode}\end{prop}
 \begin{proof}
 We have $i_j(\s)=0$ if and only if there are no entries to right of $\s_j$ that are less than $\s_j$, which is true if and only if $\s_j$ is a right-to-left minimum.
 \end{proof}

 Skandera defines $\mathrm{stc}(\s)=\st(\invcode(\s))$ in~\cite{Skandera}. It is not known to be equal to any familiar Eulerian statistic calculated directly from the permutation $\s$.

 \begin{thm}
 The triple $(\rlmin, \mathrm{stc}, \inv)$ is Stirling-Euler-Mahonian.
 \end{thm}
 \begin{proof}
 Follows directly from Theorem~\ref{thm:SEM} and Proposition~\ref{prop:zerinvcode}.
 \end{proof}

 \section{Two Stirling-Euler-Mahonian triples for $\den$}\label{Sec:Den}

 \subsection{Two definitions of $\den$}\label{Sec:IntroToDEN}

\noindent \textit{Denert's statistic} $\den$ has two equivalent definitions, both of which will be useful to us. The first definition is given by
\begin{multline}\den(\sigma)=\#\{j<k \, \vert \, \sigma_k<\sigma_j\leq k\}+\\ \#\{j<k \,\vert \, \sigma_j\leq k<\sigma_k\}+\#\{j<k \, \vert \, k<\sigma_k<\sigma_j\}.\label{eq:den1stdef}\end{multline}

The second formulation of $\den$ is easier to work with, but first requires some definitions. For a permutation $\s$, the \textit{excedance subword} $\s_E$ is the subword of all excedances tops in the order that they occur in $\s$, and the \textit{non-excedance subword} $\s_N$ is the subword of all non-excedance tops in the order that they occur in $\s$. Then
\begin{equation}\den(\s)=\inv(\s_E)+\inv(\s_N)+\sum_{\s_j>j}j.\label{eq:den2nddef}\end{equation}
Note the last term in (\ref{eq:den2nddef}) is the sum of the excedances of $\s$. See~\cite{FZ90} for a proof of the equivalence of definitions (\ref{eq:den1stdef}) and (\ref{eq:den2nddef}), along with a proof that $\den$ is Mahonian.

For example, for $\s=354162$ using the first definition of $\den$ we get $\den(354162)=6+5+1=12$. Alternately using the second definition, we have $\s_E=3546$, $\s_N=12$, $\inv(\s_E)=1$, and $\inv(\s_N)=0$, so $\den(354162)=1+0+1+2+3+5=12$.

 Denert's statistic has two known permutation codes, which lead to different results related to $\den$ and Stirling-Euler-Mahonian triples.

 \subsection{$\dencode$}
 The first code for $\den$ comes from Foata and Zeilberger~\cite{FZ90}, which we will denote $\dencode$. It is based on the definition of $\den$ given in (\ref{eq:den1stdef}). We let

\[d_j(\s)= \begin{cases}
      \#\{k<j \, \vert \, \s_j<\s_k\leq j\} & \mathrm{if} \, \s_j\leq j \\
      \#\{k<j \, \vert \, \s_k\leq j\} + \#\{k<j \, \vert \, \s_j<\s_k\} & \mathrm{if} \, \s_j>j
   \end{cases}.
\]
Note the value of $d_j(\s)$ is just the contribution of position $j$ to $\den(\s)$ from (\ref{eq:den1stdef}). Using this fact it is easy to see that $0\leq d_j(\s)\leq j-1$. We can then define $\dencode(\s)=(d_1(\s), d_2(\s),\ldots, d_n(\s))$. For example, $\dencode(341625)=(0, 0, 0, 3, 2, 1)$. This map is shown to be invertible in~\cite{FZ90}, making it clear that $\dencode:  \mathfrak{S}_n\to E_n$ is a bijection. We have the following theorem.

\begin{thm}
The triple $(\zer(\dencode), \st(\dencode), \den)$ is Stirling-Euler-Mahonian.
\end{thm}
\begin{proof}
  Follows directly from Theorem~\ref{thm:SEM}.
\end{proof}
We checked the values of the statistics $\zer(\dencode)$ and $\st(\dencode)$ on $\mathfrak{S}_n$ for $n=2, 3, 4$ in the FindStat database for combinatorial statistics~\cite{FindStat} and received no direct matches for either. This provides evidence that $\zer(\dencode)$ is not equal to any familiar Stirling statistic, nor is $\st(\dencode)$ equal to any familiar Eulerian statistic. 

\subsection{$\hancode$}
The second code for $\den$ is due to Han~\cite{Han}, so we will denote it $\hancode$. It is based on the definition of $\den$ given in (\ref{eq:den2nddef}). It has a complicated recursive definition as a function $E_n\to \mathfrak{S}_n$ (instead of $\mathfrak{S}_n\to E_n$ as we have seen with other examples above). But since the function defined is a bijection, it still represents a pairing between a permutation $\s\in \mathfrak{S}_n$ and a unique element $\hancode(\s)\in E_n$. We will describe Han's recursive process.

When $n=1$, the only permutation in $\mathfrak{S}_1$ is $\s=1$ and $\den(\s)=0$, so we have $$\hancode(\s)=0.$$

For $n\geq 2$, Han defines a function $\Psi: \mathfrak{S}_{n-1}\times \{0, 1, \ldots , n-1\}\to \mathfrak{S}_n$ which allows us to recursively find the permutation associated to a given $\hancode$ in $E_n$. Let $\s=\s_1\s_2\cdots \s_{n-1}\in \mathfrak{S}_{n-1}$. If $s=0$, then $\Psi(\s,s)=\s_1\s_2\cdots\s_{n-1}n$. 

If $s\ne 0$, first we must compute the $\nu(\s_r)$ values of $\s$ for $1\leq r\leq n-1$, where
\small\[\nu(\s_r)= \begin{cases}
      \#\{j\geq \s_r \vert j \ \textrm{is an excedance of} \ \s\} &
       \textrm{if} \ r \ \textrm{is an excedance of} \ \s \\
      \#\{j\leq \s_r \vert j \ \textrm{is a non-excedance of} \ \s\}+\exc(\s) & \textrm{if} \ r \ \textrm{is a non-excedance of} \ \s
   \end{cases}
\]
\normalsize Note the $\nu(\s_r)$ values equal to $1, 2, \ldots, \exc(\s)$ will correspond to the excedance tops of $\s$ from largest to smallest, and the $\nu(\s_r)$ values equal to $exc(\s)+1, \ldots, n$ will correspond to the non-excedance tops of $\s$ from smallest to largest. Next underline each excedance top $\s_j$ such that $\s_j\geq \nu^{-1}(s)$, and let

\begin{equation}\{\nu^{-1}(1), \nu^{-1}(2), \ldots , \nu^{-1}(m)\}\label{SetNuValues}\end{equation} 

\noindent be these underlined values. Then swap excedance tops in the following fashion: $n$ replaces $\nu^{-1}(1)$ (the largest underlined value, from the remarks above), $\nu^{-1}(1)$ replaces $\nu^{-1}(2)$, \ldots , $\nu^{-1}(m-1)$ replaces $\nu^{-1}(m)$. Finally, insert the value $p=\nu^{-1}(m)$ in position $\nu^{-1}(s)$. Then $\Psi(\s,s)$ is the permutation obtained through this procedure.

\begin{table}[]
\caption{$\nu$-values for $\s=341625$}
\label{table:nu example}
\begin{tabular}{rllllll}
$j$              & $1$ & $2$ & $3$ & $4$ & $5$ & $6$ \\
$\sigma_j$      & $3$ & $4$ & $1$ & $6$ & $2$ & $5$\\
$\nu(\sigma_j)$ & $3$ & $2$ & $4$ & $1$ & $5$ & $6$
\end{tabular}
\end{table}

Let $\s=341625$ and note that $\exc(\s)=3$. The $\nu$-values for this $\s$ are shown in Table~\ref{table:nu example}. We will consider what happens for this $\s$ in three cases:  when $s=0$, when $0<s\leq \exc(\s)$, and when $s>\exc(\s)$. 

For the first example when $s=0$, we have that $\Psi(341625,0)=3416257$. Note that $\den(\Psi(341625,0))=\den(341625)+0$ and $\exc(\Psi(341625,0))=\exc(341625)$.

The second example we consider is $\Psi(341625,2)$ (so $s=2$). Note that $2\leq \exc(341625)$. We first see that $\nu^{-1}(2)=4$, and the excedance tops of $\s$ are $3$, $4$, and $6$. Of these only $4$ and $6$ are greater than or equal to $\nu^{-1}(2)$, so we underline these values:  $3\b41\b625$. Then we put $7$ in the place of the largest excedance top $6$, $6$ in place of the second largest excedance top $4$, and then we insert $p=4$ in the position $\nu^{-1}(2)=4$. Thus $\Psi(341625,2)=3614725$. Note that $\den(\Psi(341625,2))=\den(341625)+2$ and $\exc(\Psi(341625,2))=\exc(341625)$.

The third example we consider is $\Psi(341625,5)$ (so $s=5$). Note that $5>\exc(341625)$. We first see that $\nu^{-1}(5)=2$. The excedance tops of $\s$ are $3$, $4$, and $6$, all of which are greater than or equal to $\nu^{-1}(5)$, so we underline these values:  $\b3\b41\b625$. Then we put $7$ in the place of the largest excedance top $6$, $6$ in place of the second largest excedance top $4$, $4$ in place of the smallest excedance top $3$, and then we insert $p=3$ in the position $\nu^{-1}(5)=2$, creating a new excedance in the resulting permutation. Thus $\Psi(341625,5)=4361725$. Note that $\den(\Psi(341625,2))=\den(341625)+5$ and $\exc(\Psi(341625,5))=\exc(341625)+1$.

For a permutation $\s\in\mathfrak{S}_1$, $\hancode(\s)$ is the unique sequence of $s$ values that generate $\s$ by repeatedly applying the function $\Psi$, starting with the unique permutation $\alpha=1$ in $\mathfrak{S}_n$. We previously saw that$\hancode(\alpha)=0$. It is then easy to verify that $\Psi(\alpha,0)=12$ and $\Psi(\alpha, 1)=21$. Thus $\hancode(12)=00$ and $\hancode(21)=01$. We can subsequently show that $\Psi(12, 0)=123$ so $\hancode(123)=000$, $\Psi(12,1)=312$ so $\hancode(312)=001$, etc. For reference we give all the $\hancode$ values for $\mathfrak{S}_3$ in Table~\ref{table:HancodeS_3}. 

Note that by Proposition 2.2 of~\cite{Han}, the function $\Psi$ satisfies that $\den(\Psi(\s,s))=\den(\s)+s$ (which guarantees that $\add(\hancode(\s))=\den(\s)$) and that
\[\exc(\Psi(\s,s))= \begin{cases}
      \exc(\s) & \textrm{if} \ s\leq\exc(\s) \\
      \exc(\s)+1 & \textrm{if} \ s>\exc(\s). \\
   \end{cases}
\]
It is also true by definition of $\st$ that
\[\st(\hancode(\Psi(\s,s)))= \begin{cases}
      \st(\hancode((\s))) & \textrm{if} \ s\leq\st(\hancode((\s))) \\
      \st(\hancode((\s)))+1 & \textrm{if} \ s>\st(\hancode((\s))), \\
   \end{cases}
\]
which makes it easy to prove the following Proposition.

\begin{prop}
For any $\s\in \mathfrak{S}_n$, $\st(\hancode(\s))=\exc(\s)$.\label{prop:sthancode}
\end{prop}
\begin{proof}
We'll use induction on $n$. When $n=1$ the only permutation is $\s=1$, which has $\hancode(\s)=0$. Thus $\st(\hancode(\s))=0=\exc(\s)$ and the result is true when $n=1$.

Now assume the result holds for $n\leq k$, and let $\s\in  \mathfrak{S}_{k+1}$. Then there is a $\pi\in \mathfrak{S}_k$ and $s$ with $0\leq s\leq k$ such that $\s=\Psi(\pi, s)$. By the induction hypothesis, $\st(\hancode(\pi))=\exc(\pi)$. If $s\leq \st(\hancode(\pi))=\exc(\pi)$, then 
\begin{multline*}
  \st(\hancode(\s))=\st(\hancode(\Psi(\pi,s))=\\\st(\hancode(\pi))=\exc(\pi)=\exc(\Psi(\pi,s))=\exc(\s) 
\end{multline*}
  If $s> \st(\hancode(\pi))=\exc(\pi)$, the result follows similarly.
\end{proof}

We have a simple interpretation for $\st(\hancode(\s))$ in terms of $\s$. Our goal is to find the same for $\zer(\hancode(\s))$. In doing so, we will make use of the following lemma.

\begin{table}[]
\caption{$\hancode$ values for $\mathfrak{S}_3$}
\label{table:HancodeS_3}
\begin{tabular}{cc}
$\s$ & $\hancode(\s)$ \\
123               & 000                                        \\
132               & 002                                        \\
213               & 010                                        \\
231               & 012                                        \\
312               & 001                                        \\
321               & 011                                       
\end{tabular}
\end{table}

 \begin{lem}
Let $\s$ be a permutation in $\mathfrak{S}_n$. If $\s_i$ is an excedance top of $\s$, then $\s_i$ is not a right-to-left minimum of $\s$.\label{lem:exc-rlmin}
\end{lem}

\begin{proof} If $\s_i>i$, note that $n-\s_i<n-i$. Then the $n-\s_i$ values larger than $\s_i$ cannot occupy all of the $n-i$ positions in $\s$ to the right of $\s_i$. Thus there must be a value to the right of $\s_i$ that is less than $\s_i$. Hence $\s_i$ cannot be a right-to-left minimum of $\s$.\end{proof}

 \begin{thm}
   Let $\s$ be a permutation in $\mathfrak{S}_n$. Then $\zer(\hancode(\s))=\rlmin(\s)$.\label{theorem:zerhancode}
 \end{thm}
 \begin{proof}
 We'll prove this by induction on $n$. When $n=1$ the only permutation is $\s=1$, which has $\hancode(\s)=0$. The result clearly holds in this case.

 Now assume the theorem holds for $n\leq k$, and let $\s\in S_{k+1}$. Then we know there is a $\pi\in S_k$ and $0\leq s\leq k$ such that $\s=\Psi(\pi,s)$. We also know by the induction hypothesis that $\zer(\hancode((\pi))=\rlmin(\pi)$. Let $\pi=\pi_1\pi_2\cdots \pi_k$. We'll consider three cases.

 \noindent{\bf Case 1:  $s=0$}

 By definition of $\hancode$, $\s=\pi_1\pi_2\cdots \pi_k (k+1)$ and $k+1$ is a right-to-left minimum of $\s$. Thus $\rlmin(\s)=\rlmin(\pi)+1$. Since $s=0$, we also have $\zer(\hancode((\s))=\zer(\hancode((\pi))+1$. Hence $\zer(\hancode((\s))=\rlmin(\s)$ in this case.

 \noindent{\bf Case 2:  $0<s\leq\exc(\pi)$}

  Since $s>0$, we have $\zer(\hancode((\s))=\zer(\hancode((\pi))$. By the induction hypothesis, $\zer(\hancode((\pi))=\rlmin(\pi)$. By definition of the Han code, $\nu^{-1}(s)$ is the $s$th largest excedance top of $\pi$. Note that since $0<s\leq\exc(\pi)$, the value $m$ in (\ref{SetNuValues}) is equal to $s$. In order to produce $\s$, $k+1$ replaces the largest excedance top of $\pi$, the largest excedance top replaces the second largest excedance top, continuing in this way until $\nu^{-1}(s)$ is reached, and $p=\nu^{-1}(s)$ is inserted in position $p$. None of these excedances tops in $\s$ can be right-to-left minima by Lemma~\ref{lem:exc-rlmin}. The relative positions of the non-excedances tops of $\pi$ stay the same in $\s$. The value of each excedance top of $\pi$ is increased in $\s$, which has no effect on the right-to-left minima in the non-excedance positions. So the only possible way that the number of right-to-left minima of $\s$ could increase over that of $\pi$ is if $p$ were a right-to-left minimum of $\s$. We will show that this is not possible.
 
 Since $p$ is position $p$ of $\s$, the only way it could be a right-to-left minimum of $\s$ is if the values $p+1, p+2, \ldots , k+1$ were all to the right of it. The value $p$ is an excedance top in the permutation $\pi$, so it must be to the left of position $p$ in $\pi$. In the swapping algorithm of $\Psi$ that produces $\s$, $p$ is replaced with $\nu^{-1}(s-1)$, which is the next larger excedance top of $\pi$. Thus we must have $\nu^{-1}(s-1)>p$ and $\nu^{-1}(s-1)$ must be in position $p-1$ or smaller in $\s$. Thus there is a value to the left of $p$ in $\s$ that is larger than $p$. So not all of the values $p+1, p+2, \ldots , k+1$ can be to the right of $p$ in $\s$. This fact forces there to be a value to the right of $p$ in $\s$ that is smaller than $p$. Thus $p$ cannot be a right-to-left minimum of $\s$. Therefore, $\rlmin(\s)=\rlmin(\pi)$ and hence $\zer(\hancode((\s))=\rlmin(\s)$ as desired.

  \noindent{\bf Case 3:  $s> \exc(\pi)$}

  Again $s>0$ implies $\zer(\hancode((\s))=\zer(\hancode((\pi))$, and the induction hypothesis implies $\zer(\hancode((\pi))=\rlmin(\pi)$. In this case, $\s$ is formed by replacing the largest excedance top of $\pi$ with $k+1$, the second largest excedance top of $\pi$ with the largest excedance top, etc. until the $(s-1)$th largest excedance top of $\pi$ replaces the $s$th largest excedance top, and the $s$th largest excedance top of $\pi$ is inserted at position $\nu^{-1}(s)$ to create a new excedance at this position. Again, none of the now $s+1$ excedances of $\s$ can be right-to-left minima by Lemma~\ref{lem:exc-rlmin}. The relative positions of the non-excedances of $\pi$ stay the same in $\s$. The value of each excedance of $\pi$ is increased in $\s$, which has no effect on the right-to-left minima in the non-excedance positions. Thus there is no change in the right-to-left minima from $\pi$ to $\s$, hence $\rlmin(\s)=\rlmin(\pi)$ and $\zer(\hancode((\s))=\rlmin(\s)$ as desired.

 \end{proof}

 \begin{thm} The triple $(\rlmin,\exc,\den)$ is Stirling-Euler-Mahonian.\label{thm:rlminexcdenSEM}\end{thm}
 \begin{proof}
   Follows directly from Theorem~\ref{thm:SEM}, Proposition~\ref{prop:sthancode}, and Theorem~\ref{theorem:zerhancode}.
 \end{proof}

 \section{A Stirling-Euler-Mahonian triple for $\mak$}\label{Sec:IntroToMAK}
The statistic $\mak$ was first introduced in~\cite{FZ90}, where a proof that it is Mahonian can be found. We will use the definition of $\mak$ from~\cite{CSZ97}, which first requires some terminology. A \textit{descent block} of a permutation $\s\in\mathfrak{S}_n$ is a maximal decreasing subword of $\sigma$. A descent block is an \textit{outsider} if it consists of only one value, otherwise it is called a \textit{proper descent block}. In a proper descent block, the leftmost (and largest) value is called its \textit{closer}, and the rightmost (and smallest) value is called its \textit{opener}. A descent block \textit{embraces} $\sigma_j$ of $\sigma$ if the value $\sigma_j$ is strictly between the opener and closer of the block. The \textit{right embracing number} of a value $\sigma_j$ of $\sigma$, denoted $\rem(\s_j)$, is the number of descent blocks strictly to the right of $\s_j$ in $\s$ that embrace $\s_j$. We can then define
$$\mak(\s)=\sum_{\s_{j+1}>\s_{j}} \s_{j} + \sum_{j=1}^{n} \rem(\s_j).$$
Note the first term is just the sum of all the descent bottoms of $\s$.

For example, consider the permutation $\s=354162$, which can be written with dashes separating the descent blocks as $3-541-62$ to make calculations easier. We can now see that $\s$ has descent blocks $3$, $541$, and $62$ and descent bottoms $1$, $2$, and $4$. We can see that $\rem(1)=0$, $\rem(2)=0$, $\rem(3)=2$, $\rem(4)=1$, $\rem(5)=1$, and $\rem(6)=0$, so $\mak(354162)=1+2+4+2+1+1=11$.
 
 Our study of $\mak$ requires a different approach, since there was previously no known permutation code for this statistic. Along the way, we will indirectly define $\makcode$. For a permutation $\s\in\mathfrak{S}_n$, we define the \textit{inversion bottom number} of a value $\s_j$ as $\#\{\s_k\ \vert \ k<j \ \mathrm{and} \ \s_k>\s_j\}$. This counts the number of inversions in $\s$ where $\s_j$ is the smaller value. We also define the \textit{inversion top number} of $\s_j$ as $\#\{\s_k \ \vert \ k>j \ \mathrm{and} \  \s_j>\s_k\}$. This counts the number of inversions in $\s$ where $\s_j$ is the larger value. Finally, we will call a permutation written as a function in two-line notation a \textit{biword}.

 In~\cite{CSZ97}, the authors define a bijection $\Phi: \mathfrak{S}_n\to \mathfrak{S}_n$ and prove that it satisfies that for all $\s\in \mathfrak{S}_n$, \begin{equation}\des(\s)=\exc(\Phi(\s)) \ \textrm{and} \ \mak(\s)=\den(\Phi(\s)).\label{eq:phiproperties}\end{equation} The existence of $\Psi$ implies that $(\des,\mak)$ is Euler-Mahonian. In addition, it is proven that the set of descent tops of $\s$ equals the set of excedance tops of $\Phi(\s)$ and the set of descent bottoms of $\s$ equals the set of excedances of $\Phi(\s)$. Note this bijection also preserves other permutation statistics which we will not use here. The bijection is defined as follows.

For a permutation $\s\in \mathfrak{S}_n$, we will form the biwords $\begin{pmatrix}f \\f^{\prime}\end{pmatrix}$ and $\begin{pmatrix}g \\g^{\prime}\end{pmatrix}$, then construct the biword $\tau^{\prime}=\begin{pmatrix}f & g \\f^{\prime} & g^{\prime}\end{pmatrix}$ by concatenating $f$ and $g$, and $f^{\prime}$ and $g^{\prime}$, respectively. The word $f$ is the subword of descent bottoms of $\s$ in increasing order, and the word $g$ is the subword of non-descent bottoms of $\s$, also in increasing order. The word $f^{\prime}$ is the subword of descent tops of $\s$, ordered so that the inversion bottom number of a value $k$ in $f^{\prime}$ equals the right embracing number of $k$ in $\s$, and $g^{\prime}$ is the subword of non-descent tops in $\s$, ordered so that the inversion top number of a value $\ell$ in $g^{\prime}$ equals the right embracing number of $\ell$ in $\s$. It is proven in~\cite{CSZ97} that such a configuration of $f^{\prime}$ and $g^{\prime}$ is both possible and uniquely defined for each $\s\in \mathfrak{S}_n$. After reordering the columns of $\tau^{\prime}$ so that the top row is in increasing order, we take the bottom row of the rearranged biword to be $\tau=\Phi(\s)$.

 For example if $\s=354162$, then , then $f=124$ and $g=356$. We previously calculated that $\rem(1)=0$, $\rem(2)=0$, $\rem(3)=2$, $\rem(4)=1$, $\rem(5)=1$, and $\rem(6)=0$, so $f^{\prime}=645$ and $g^{\prime}=312$. Thus $\tau^{\prime}=\begin{pmatrix}1 & 2 & 4 & 3 & 5 & 6 \\6 & 4 & 5 & 3 & 1 & 2\end{pmatrix}$, and the bottom row obtained after reordering the top row yields $\tau=643512$. Note that the descent tops of $\s$ are $4, 5$, and $6$, which are the excedance tops of $\tau$, the descent bottoms of $\s$ are $1, 2$, and $4$, which are the excedances of $\tau$, and that $\mak(\s)=11=\den(\tau)$ and $\des(\s)=3=\exc(\tau)$.

 The following theorem establishes a Stirling partner for the Euler-Mahonian pair $(\des, \mak)$.

 \begin{thm}Let $\Phi:\mathfrak{S}_n\to \mathfrak{S}_n$ be the bijection described above and suppose $\sigma\in \mathfrak{S}_n$. Then $\rlmin(\s)=\rlmin(\Phi(\s))$.\label{thm:phirlmin}\end{thm}
 \begin{proof}
 Let $\tau=\Phi(\s)$. Suppose $\s_i$ is a right-to-left minimum of $\s$. Since $\s_i$ is smaller than all value to its right in $\s$, in particular $\s_{i+1}>\s_i$. Thus $\s_i$ is the closer for a (possibly outsider) descent block in $\s$, and hence $\s_i$ must be a non-descent top of $\s$. Therefore the image $\tau_j$ of $\s_i$ under the map $\Phi$ is in the subword $g^{\prime}$. Additionally, $\s_i$ being smaller than every entry to its right implies that no descent block to the right of $\s_i$ can embrace $\s_i$. Thus the right embracing number of $\s_i$ in $\s$ is $0$. Hence the location of $\tau_j$ in the subword $g^{\prime}$ under the map $\Phi$ is chosen as to have inversion top number equal to $0$. This implies that no entry in $\tau^{\prime}$ to the right of $\tau_j$ is less than $\tau_j$, hence $\tau_j$ is a right-to-left minimum of $\tau^{\prime}$. We will now show that $\tau_j$ is also a right-to-left minimum of $\tau$. 
 
 We know $\tau_j$ must be in the subword $g^{\prime}$, but more generally all the right-to-left minima of $\tau^{\prime}$ must be in the subword $g^{\prime}$. The right-to-left minima in a permutation must appear in ascending order when reading the permutation word from left to right, since a right-to-left minimum cannot have a smaller right-to-left minimum to its right. So the right-to-left minima in $g^{\prime}$ must appear in ascending order. Upon placing the top row of $\tau^{\prime}$ in ascending order to obtain $\tau$, because the entries in $g$ are in ascending order it is clear that the values of $g^{\prime}$ will remain in the same order relative to one another in $\tau$. The only way $\tau_j$ could fail to be a right-to-left minimum of $\tau$ would be if there were a right-to-left minimum smaller than $\tau_j$ to the right of $\tau_j$ in the word for $\tau$. This cannot happen with any values from the subword $g^{\prime}$ for the reasons given above. All of the values in $f^{\prime}$ are excedance tops in $\tau^{\prime}$, and will remain excedance tops when reordered to produce $\tau$. So none of the values from $f^{\prime}$ can be right-to-left minima to the right of $\tau_j$, again by Lemma~\ref{lem:exc-rlmin}. Hence $\tau_j$ has no right-to-left minimum smaller than $\tau_j$ to its right in the word for $\tau$. Thus $\tau_j$ is a right-to-left minimum of $\tau$.       
 

 Now suppose $\tau_j$ is a right-to-left minimum of $\tau$. By Lemma~\ref{lem:exc-rlmin}, $\tau_j$ must be a non-excedance top of $\tau$. Also since $\tau_j$ is less than all entries to its right, the inversion top number of $\tau_j$ is $0$. Hence under the bijection $\Phi$, $\tau_j$ is the image of an entry $\s_i$ that is a non-descent top with right embracing number $0$. Since $\s_i$ is a non-descent top, note we must have that $\tau_j$ is in the word $g^{\prime}$ and $\s_{i+1}>\s_i$. Let $\s_k$ be the opener of the descent block with closer $\s_{i+1}$. But then we must have $\s_k>\s_i$, otherwise the descent block will embrace $\s_i$, contradicting that the right embracing number of $\s_i$ is $0$. We also must have that $\s_{k+1}>\s_k$ since $\s_{k+1}$ begins a new descent block. But then we also have that $\s_{k+1}>\s_i$, so the opener of the descent block with closer $\s_{k+1}$ must also be greater than $\s_i$, otherwise this descent block will embrace $\s_i$. We continue in this manner until we consider all entries to the right of $\s_i$, all of which must be greater than $\s_i$. Thus $\s_i$ is a right-to-left minimum of $\s$.

 Since the right-to-left minima of $\s$ are in one-to-one correspondence with the right-to-left minima of $\tau=\Phi(\s)$, we conclude that $\rlmin(\s)=\rlmin(\Phi(\s))$.
 \end{proof}

 \begin{thm} The triple $(\rlmin, \des, \mak)$ is Stirling-Euler-Mahonian.
 \end{thm}
 \begin{proof}
   Follows from Theorem~\ref{thm:rlminexcdenSEM}, (\ref{eq:phiproperties}), and Theorem~\ref{thm:phirlmin}.
 \end{proof}

 \begin{table}[]
 \caption{$\makcode$ for $\mathfrak{S}_3$}
 \label{table:MAKCodeS_3}
\begin{tabular}{ccc}
$\s$ & $\Phi(\s)$ & $\makcode(\s)=\hancode(\Phi(\s))$ \\
123  & 123        & 000                                                                                                          \\
132  & 132        & 002                                                                                                          \\
213  & 213        & 010                                                                                                          \\
231  & 321        & 011                                                                                                          \\
312  & 312        & 001                                                                                                          \\
321  & 231        & 012                                                                                                         
\end{tabular}
\end{table}

 Note that we can use the map $\Phi$ to define $\makcode(\s)=\hancode(\Phi(\s))$. For reference, the values of $\makcode$ are shown for $\mathfrak{S}_3$ in Table~\ref{table:MAKCodeS_3}. Using Proposition~\ref{prop:sthancode}, Theorems~\ref{theorem:zerhancode} and~\ref{thm:phirlmin}, and (\ref{eq:phiproperties}), we can easily see that $\st(\makcode(\s))=\des(\s)$ and $\zer(\makcode(\s))=\rlmin(\s)$. The author tried several obvious approaches to define $\makcode$ directly from $\s$, but this definition was the only one that was found to work.

 \section{A Stirling-Euler-Mahonian triple for $\sor$}\label{Sec:IntroToSor}
   One of the most recently developed Mahonian statistics is the \textit{sorting index} $\sor$. In order to define it, let $\s\in\mathfrak{S}_n$ and let $0\leq j\leq n$. Assume $j$ appears in cycle $C$ of the disjoint cycle decomposition of $\s$. If $j$ is the smallest value in the cycle $C$, we define $c_j(\s)=j$. If $j$ is not the smallest value in the cycle $C$, then we define $c_j(\s)$ to be the value $k<j$ that is the smallest number of steps away from $j$ when moving from $j$ to $k$ along the cycle $C$. We then define the sorting index by the equation
   $$\sor(\s)=\sum_{j=1}^{n} (j-c_j).$$
   \noindent See~\cite{Petersen, Wilson} for more on $\sor$, including a proof that it is Mahonian. As an example, we see that $\s=354162$ has disjoint cycle decomposition $(134)(256)$, so $c_1(\s)=1$, $c_2(\s)=2$, $c_3(\s)=1$, $c_4(\s)=1$, $c_5(\s)=2$, and $c_6(\s)=2$. Hence
   $\sor(354162)=(1-1)+(2-2)+(3-1)+(4-1)+(5-2)+(6-2)=12$.
 
 It is clear from the definition that $1\leq c_j\leq j$ for all $j$, so $0\leq j-c_j\leq j-1$. We let
 $$s_j(\s)=j-c_j(\s)$$
 and define $\sorcode(\s)=(s_1(\s),s_2(\s),\ldots, s_n(\s))$. This is the $B$-code of~\cite{CGG}, where it is proven that $\sorcode:\mathfrak{S}_n\to E_n$ is a bijection. For example, $\sorcode(341625)=(0,0,2,2,3,1)$.

 \begin{prop}
    Let $\s$ be a permutation in $\mathfrak{S}_n$. Then $\zer(\sorcode(\s))=\cyc(\s)$.\label{prop:zersorcode}
 \end{prop}
\begin{proof}
If $j$ is the smallest value in its cycle $C$, then we get that $c_j(\s)=j$ and thus $s_j(\s)=0$. If $j$ is not the smallest value in its cycle then $c_j(\s)<j$, so $s_j(\s)>0$. Hence $\zer(\sorcode(\s))=\cyc(\s)$ as desired.
\end{proof}

\begin{thm}
The triple $(\cyc, \st(\sorcode),\sor)$ is Stirling-Euler-Mahonian.
\end{thm}
\begin{proof}
  Follows from Theorem~\ref{thm:SEM} and Proposition~\ref{prop:zersorcode}.
\end{proof}

We checked the values of the statistic $\st(\sorcode)$ on $\mathfrak{S}_n$ for $n=2, 3, 4$ in the FindStat database for combinatorial statistics~\cite{FindStat} and received no direct matches. This provides evidence that $\st(\sorcode)$ is not equal to any familiar Eulerian statistic. 

\section{Properties of Stirling-Euler-Mahonian triples}\label{Sec:Properties}
In this paper we have established that the following triples of permutation statistics are Stirling-Euler-Mahonian:  \begin{itemize} \item $(\rlmin, \mathrm{stc},\inv)$, \item $(\zer(\dencode), \st(\dencode), \den)$, \item $(\rlmin, \exc, \den)$, \item $(\rlmin, \des, \mak)$, and \item $(\cyc, \st(\sorcode), \sor)$.\end{itemize}
In addition, it was shown that the triple of statistics $$(\cyc(\mathrm{DG}),\des,\mathrm{s}_{\mathbb{T}_n, \mathrm{b}}(\mathrm{DG})+\mathrm{E}(\mathrm{DG}))$$ described in~\cite{CycleCounting} is Stirling-Euler-Mahonian. The results in this section apply to these or any other Stirling-Euler-Mahonian triples.

Let $f(q)=\sum_{j=M}^{N} a_jq^j$ be a polynomial in $q$ with $a_M, a_N\ne 0$. We call $M+N$ the \textit{virtual degree} of $f(q)$. The polynomial $f(q)$ is called \textit{symmetric} provided $a_{N-j}=a_{M+j}$ for $0\leq j\leq N-M$, and \textit{unimodal} provided there exists a $P$, $M\leq P\leq N$, such that $a_M\leq\cdots\leq a_{P}\geq a_{P+1}\geq \cdots \geq a_N.$ Recall the \textit{$q$-analog} of the natural number $m$ is given by $\br{m}=1+q+\cdots +q^{m-1}$.

\begin{thm} Let $(\sstat, \estat, \mstat)$ be any Stirling-Euler-Mahonian triple of permutation statistics. Then for any $m\in\N$, the polynomial
$$\sum_{\s\in \mathfrak{S}_n, \ \estat(\s)=k-1} \br{m}^{\sstat(\s)}q^{(n-\sstat(\s))(m-1)+\mstat(\s)} $$
in $q$ is symmetric and unimodal with virtual degree $n(m+k-2)+(k-1)(m-1)$.
\end{thm}
\begin{proof}
In~\cite{SymmUni} it was proven that
$$\sum_{\s\in \mathfrak{S}_n, \ \des(\s)=k-1} \br{m}^{\rlmin(\s)}q^{(n-\rlmin(\s))(m-1)+\maj(\s)} $$
is symmetric and unimodal with virtual degree $n(m+k-2)+(k-1)(m-1)$. But $(\sstat, \estat, \mstat)$ is equidistributed with $(\rlmin, \des, \maj)$, so the result holds.
\end{proof}

For the next theorem, we need to define several additional $q$-analogs. For any real number $y$, the \textit{$q$-analog of the real number $y$} is given by $$\br{y}=\frac{1-q^y}{1-q}.$$ Note this generalizes the definition of the $q$-analog of a natural number. For $m\in\N$, the \textit{$q$-analog of $m!$} is given by $\br{m}!=\br{1}\br{2}\cdots\br{m-1}\br{m}$. The \textit{$q$-analog of the binomial coefficient} is
$${k \brack j}=\frac{\br{k}\br{k-1}\cdots \br{k-j+1}}{\br{j}!}.$$
Note this definition makes sense even if $k$ is a real number.
\begin{thm} Let $(\sstat, \estat, \mstat)$ be any Stirling-Euler-Mahonian triple of permutation statistics. Then for any $n, k\in\N$ and $1\leq k\leq n$ we have
\begin{multline*}\sum_{\s\in \mathfrak{S}_n, \ \estat(\s)=k-1} \br{y}^{\sstat(\s)}q^{(n-\sstat(\s))(y-1)+\mstat(\s)}=\\ \sum_{j=0}^{k-1}{y+n \brack k-j-1}{y+j-1 \brack j}(-1)^{k-j-1}q^{\binom{k-j-1}{2}}\br{y+j}^n.\end{multline*}
\end{thm}
\begin{proof}
In~\cite{mythesis} it was proven that
\begin{multline*}\sum_{\s\in \mathfrak{S}_n, \ \des(\s)=k-1} \br{y}^{\rlmin(\s)}q^{(n-\rlmin(\s))(y-1)+\maj(\s)}=\\ \sum_{j=0}^{k-1}{y+n \brack k-j-1}{y+j-1 \brack j}(-1)^{k-j-1}q^{\binom{k-j-1}{2}}\br{y+j}^n.\end{multline*}
But $(\sstat, \estat, \mstat)$ is equidistributed with $(\rlmin, \des, \maj)$, so the result holds.
\end{proof}

In~\cite{CycleCounting} the \textit{$q,y$-Eulerian numbers} $\qyeul{n}{k}$ are defined as
\begin{equation}\qyeul{n}{k}=\sum_{\s\in \mathfrak{S}_n, \ \des(\s)=k-1}\br{y}^{\rlmin(\s)}q^{(n-\rlmin(\s))(y-1)+\maj(\s)}.\label{eq:qyEulerianNums}\end{equation}
But since $(\sstat, \estat, \mstat)$ is equidistributed with $(\rlmin, \des, \maj)$ if $(\sstat, \estat, \mstat)$ is any Stirling-Euler-Mahonian triple, we could actually define the $q,y$-Eulerian numbers as
$$\qyeul{n}{k}=\sum_{\s\in \mathfrak{S}_n, \ \estat(\s)=k-1}\br{y}^{\sstat(\s)}q^{(n-\sstat(\s))(y-1)+\mstat(\s)}$$
and obtain the following recurrence.
\begin{thm}
Let $\qyeul{n}{k}$ be as defined in (\ref{eq:qyEulerianNums}) for any Stirling-Euler-Mahonian triple $(\sstat, \estat, \mstat)$. Then for any $n, k\in\N$ we have
$$\qyeul{n}{k}=\br{y+k-1}\qyeul{n-1}{k}+q^{y+k-2}\br{n-k+1}\qyeul{n-1}{k-1}.$$
\end{thm}
\begin{proof}
Since $(\sstat, \estat, \mstat)$ is equidistributed with $(\rlmin, \des, \maj)$, we know that
$$\qyeul{n}{k}=\sum_{\s\in \mathfrak{S}_n, \ \des(\s)=k-1}\br{y}^{\rlmin(\s)}q^{(n-\rlmin(\s))(y-1)+\maj(\s)}.$$
In~\cite{CycleCounting} it was shown that
$$\sum_{\s\in \mathfrak{S}_n, \ \des(\s)=k-1}\br{y}^{\rlmin(\s)}q^{(n-\rlmin(\s))(y-1)+\maj(\s)}$$
satisfies the desired recurrence, so we get that $\qyeul{n}{k}$ does as well.
\end{proof}


\bibliography{StirlingEulerMahonian.bib}{}
\bibliographystyle{plain}

\end{document}